\documentclass[oneside,10pt]{article}
\usepackage[b5paper]{geometry}	

\usepackage{amsfonts,amsmath,latexsym,amssymb}
\usepackage{mathrsfs,upref}
\usepackage{mathptmx}		    	                            		
\usepackage{amscd}
\usepackage{amsthm}
\usepackage{graphicx}
\usepackage{amsmath}
\usepackage{lipsum}

\newcommand\blfootnote[1]{%
  \begingroup
  \renewcommand\thefootnote{}\footnote{#1}%
  \addtocounter{footnote}{-1}%
  \endgroup}

\newtheorem{theorem}{Theorem}

\newtheorem{lemma}[theorem]{Lemma}

\newtheorem{proposition}[theorem]{Proposition}

\begin{document}

\title{A remark about Calder\'on-Hardy spaces with variable exponents}
\author{Pablo Rocha}

\maketitle

\begin{abstract}
In this note we improve the parameter $q$ that appears in Theorem 1 obtained by the author in [Math. Ineq. \& appl., Vol 19 (3) (2016), 
1013-1030].
\end{abstract}

\blfootnote{{\bf Keywords}: Variable Calder\'on-Hardy Spaces, Variable Hardy Spaces, Atomic decomposition. \\
{\bf 2020 Mathematics Subject Classification:} 42B25, 42B30}

\textbf{Notation and terminology:}

\

- The symbol $A\lesssim B$ stands for the inequality $A \leq cB$ for some positive constant $c$, and $A \approx B$ stands for 
$B \lesssim A \lesssim B$. \\

- We denote by $Q\left( x_0,r\right)$ the cube centered at $x_0 \in \mathbb{R}^{n}$ with side lenght $r$. Given a cube 
$Q = Q(x_0,r )$, we set $\delta Q = Q(x_0, \delta r)$. \\

- For a measurable subset $E\subseteq \mathbb{R}^{n}$, we denote by $\left\vert E\right\vert$ and $\chi_{E}$ the Lebesgue measure of $E$ and the characteristic function of $E$ respectively. \\

- $M$ denotes the Hardy-Littlewood maximal operator given by
\[
Mf(x) = \sup_{Q \ni x} |Q|^{- 1}\int_{Q} |f(y)| \, dy,
\]
where $f$ is a locally integrable function on $\mathbb{R}^{n}$ and the supremum is taken over all the cubes $Q$ containing $x$. \\

- $\Delta^m$ stands for the iterated Laplacian on $\mathbb{R}^n$. \\

- A measurable function $p(\cdot) : \mathbb{R}^n \to (0, \infty)$ is called exponent function or only exponent, and set 
$p_{-}= \inf_{x \in \mathbb{R}^{n}} p(x)$, $p_{+} = \sup_{x \in \mathbb{R}^{n}} p(x)$ and $\underline{p} = \min \{ p_{-}, 1 \}$. \\

- We say that an exponent $p(\cdot)$ is locally log-H\"{o}lder continuous, and denote this by $p(\cdot)\in LH_{0}(\mathbb{R}^{n})$, if there exists a positive constant $C_{0}$ such that
\[
\left\vert p(x)-p(y)\right\vert \leq\frac{C_{0}}{-\log\left\vert
x-y\right\vert }, \,\,\, \left\vert
x-y\right\vert <\frac{1}{2}.
\]

- We say that an exponent $p(\cdot)$ is log-H\"{o}lder continuous at infinity, and denote this by $p(\cdot)\in LH_{\infty}(\mathbb{R}^{n})$, if there exists a positive constant $C_{\infty}$ such that
\[
\left\vert p(x)-p(y)\right\vert \leq\frac{C_{\infty}}{\log\left(  e+\left\vert
x\right\vert \right)  }, \,\,\, \left\vert
y\right\vert \geq \left\vert x \right\vert.
\]

- $\left( L^{p(\cdot)}(\mathbb{R}^n), \| \cdot \|_{L^{p(\cdot)}} \right)$ is the Lebesgue space with variable exponents on $\mathbb{R}^n$, with
\[
\| f \|_{L^{p(\cdot)}} = \inf \left\{ \lambda > 0 : \int_{\mathbb{R}^{n}} \, \left| \frac{f(x)}{\lambda} \right|^{p(x)} \, dx \leq 1 \right\}.
\] 

- $\left( H^{p(\cdot)}(\mathbb{R}^n), \| \cdot \|_{H^{p(\cdot)}} \right)$ is the Hardy space with variable exponents on $\mathbb{R}^n$ (see \cite{Nakai}). \\

- A function $a(\cdot)$ on $\mathbb{R}^{n}$ is called an $(p(\cdot),p_{0},d_{p(\cdot)})$ - atom, if there exists a cube $Q$ such that\newline
$a_{1})$ $\textit{supp}\left( a\right) \subset Q,$\newline
$a_{2})$ $\left\Vert a\right\Vert _{p_{0}}\leq \frac{\left\vert Q\right\vert
^{\frac{1}{p_{0}}}}{\left\Vert \chi _{Q}\right\Vert _{p(\cdot)}}$, $0 < p_{-} \leq  p_{+} < p_{0} \leq \infty $ and $p_{0} \geq 1$, \newline
$a_{3})$ $\int a(x)x^{\alpha }dx=0$ for all $\left\vert \alpha \right\vert
\leq d_{p(\cdot)} := \min \left\{ l\in \mathbb{N\cup }\left\{ 0\right\} : p_{-}(n+l+1)>n\right\}.$ \\

- For sequences of nonnegative numbers $\left\{ k_{j}\right\} _{j=1}^{\infty }$
and cubes $\left\{ Q_{j}\right\} _{j=1}^{\infty }$ and for an exponent $p(\cdot):
\mathbb{R}^{n}\rightarrow \left( 0,\infty \right) $, we define
\[
\mathcal{A}\left( \left\{ k_{j}\right\}_{j=1}^{\infty },\left\{Q_{j}\right\}_{j=1}^{\infty }, p(\cdot)\right) =
\left\Vert \left\{\sum\limits_{j=1}^{\infty }\left( \frac{k_{j}\chi _{Q_{j}}}{\left\Vert \chi_{Q_{j}}\right\Vert_{p(\cdot)}} 
\right)^{\underline{p}}\right\}^{\frac{1}{\underline{p}}}\right\Vert_{p(\cdot)}.
\]

- Every $f \in H^{p(\cdot)}(\mathbb{R}^n)$ admits an atomic decomposition $f=\sum\limits_{j=1}^{\infty }k_{j}a_{j}$ (see \cite{Nakai}), where 
$\left\{ k_{j}\right\}_{j=1}^{\infty }$ is a sequence of non negative numbers, the $a_{j}$'s are $(p(\cdot),p_{0},d)$ - atoms and 
\[
\mathcal{A}\left( \left\{ k_{j}\right\}_{j=1}^{\infty },\left\{ Q_{j}\right\} _{j=1}^{\infty }, p(\cdot)\right) \lesssim 
\|f \|_{H^{p(\cdot)}}.
\]

- $\left( \mathcal{H}^{p(\cdot)}_{q, 2m}(\mathbb{R}^{n}), \| \cdot \|_{\mathcal{H}^{p(\cdot)}_{q, 2m}} \right)$ is the Calder\'on-Hardy space with variable exponents on $\mathbb{R}^n$ (see \cite{Rocha}).

\

In \cite{Rocha}, we proved the following result.

\begin{theorem} Let $p(\cdot)$ be an exponent that belongs to $LH_{0}(\mathbb{R}^{n})\cap LH_{\infty}(\mathbb{R}^{n})$, 
$1 < q < \infty$ and $m \in \mathbb{N}$ such that $0 < p_{-} \leq p_{+} < \infty$ and $n (2m + n/q)^{-1} < \underline{p}$. Then 
for $q$ sufficiently large the operator $\Delta^{m}$ is a bijective mapping from $\mathcal{H}^{p(\cdot)}_{q, 2m}(\mathbb{R}^{n})$ onto 
$H^{p(\cdot)}(\mathbb{R}^{n})$. Moreover, there exist two positive constant $c_1$ and $c_2$ such that
\[
c_1 \|F \|_{\mathcal{H}^{p(\cdot)}_{q, 2m}} \leq \| \Delta^{m}F \|_{H^{p(\cdot)}} \leq c_2 \|F \|_{\mathcal{H}^{p(\cdot)}_{q, 2m}}
\]
hold for all $F \in \mathcal{H}^{p(\cdot)}_{q, 2m}(\mathbb{R}^{n})$.
\end{theorem}

In others words, Theorem 1 says that the equation
\[
\Delta^m F = f, \,\,\,\,\,\, \text{for} \,\, f \in H^{p(\cdot)}(\mathbb{R}^n),
\]
has a unique solution $F$ in $\mathcal{H}^{p(\cdot)}_{q, 2m}(\mathbb{R}^{n})$, for $1 < q < \infty$ ({\bf sufficiently large}) and 
$m \in \mathbb{N}$ such that $n (2m + n/q)^{-1} < \underline{p}$.

\

If $n$ is large, and since the parameter $q$ is sufficiently large in Theorem 1, then one can be forced to take $m$ also large to satisfy the condition $n (2m + n/q)^{-1} < \underline{p} = \min \{ p_{-}, 1 \}$. For instance, when $n$ is large and $n/q$ is small. This restricts the possible values of $m$ once fixed $n$ large. In this note, we will remove the condition that $q$ be sufficiently large.

\

Next, we give the main steps to improve the parameter $q$ in Theorem 1.

\

Following the proof of Theorem 1 in \cite[p. 1026]{Rocha}, we compute the $\| \cdot \|_{L^{p(\cdot)}}$ - norm of the following pointwise inequality
\begin{eqnarray*}
\sum_{j=1}^{\infty} k_j N_{q, 2m}(B_j; x) &\lesssim& \sum_{j=1}^{\infty} k_j \frac{\left[M(\chi_{Q_j})(x) \right]^{\frac{2m + n/q - \mu}{n}}}{\| \chi_{Q_j} \|_{p(\cdot)}} + \sum_{j=1}^{\infty} k_j \chi_{4\sqrt{n}Q_j}(x) M(a_j)(x) \\
&+& \sum_{j=1}^{\infty} k_j \chi_{4 \sqrt{n} Q_j}(x) [M(M^{q}(a_j))(x)]^{1/q} \\
&+& \sum_{j=1}^{\infty} k_j \chi_{4 \sqrt{n} Q_j}(x)  \sum_{|\alpha|=2m} T^{*}_{\alpha}(a_j) (x) \\
&=& I + II + III + IV,
\end{eqnarray*}
where the $a_j$'s are $(p(\cdot),p_{0},d_{p(\cdot)})$ - atoms corresponding to an atomic decomposition of an arbitrary element 
$f \in H^{p(\cdot)}(\mathbb{R}^n)$, and each $a_j$ is supported on $Q_j$. 

We observe that only need to improve the estimation of $III$, where
\[
III = \sum_{j=1}^{\infty} k_j \chi_{4 \sqrt{n} Q_j}(x) [M(M^{q}(a_j))(x)]^{1/q}.
\] 
Indeed, in \cite{Rocha} to estimate $III$ we take $q > 1$ sufficiently large such that 
$\delta = \frac{1}{q}$ satisfies the hypothesis of Lemma 4.11 in \cite{Nakai}. This is, $q$ must be such that 
$\frac{1}{q} \in (0, -\frac{\log_{2}(\beta)}{n+1})$, where $\beta$ is an unspecified constant of $(0,1)$. In the estimates of $I$, $II$ and 
$IV$ it is not required $q$ large.

\

To improve the estimate of $III$, we will need two supporting results, the first is a version of \cite[Lemma 5.4]{Ho} obtained by K.-P. Ho, and the second one refers to the amount $\mathcal{A}\left( \left\{ k_{j}\right\}_{j=1}^{\infty },\left\{Q_{j}\right\}_{j=1}^{\infty }, 
p(\cdot)\right)$.

\begin{proposition} \label{b_k functions}
Let $p(\cdot) : \mathbb{R}^{n} \to (0, \infty)$ such that $p(\cdot) \in LH_{0} \cap LH_{\infty}(\mathbb{R}^{n})$ and 
$0 < p_{-} \leq p_{+} < \infty$. Let $s > 1$ and $0 < p_{*} < \underline{p}$ such that $s p_{*} > p_{+}$ and let 
$\{ b_j \}_{j=1}^{\infty}$ be a sequence of nonnegative functions in $L^{s}(\mathbb{R}^{n})$ such that each $b_j$ is supported 
in a cube $Q_j \subset \mathbb{R}^{n}$ and
\begin{equation} \label{bks}
\| b_j \|_{L^{s}(\mathbb{R}^{n})} \leq A_j |Q_j|^{1/s},
\end{equation}
where $A_j >0$ for all $j \geq 1$. Then, for any sequence of nonnegative numbers $\{ k_j \}_{j=1}^{\infty}$ we have
\[
\left\| \sum_{j=1}^{\infty} k_j b_j \right\|_{L^{p(\cdot)/p_{*}}(\mathbb{R}^{n})} \leq C \left\| \sum_{j=1}^{\infty} A_j k_j 
\chi_{Q_j} \right\|_{L^{p(\cdot)/p_{*}}(\mathbb{R}^{n})},
\]
where $C$ is a positive constant which does not depend on $\{ b_j \}_{j=1}^{\infty}$, $\{ A_j \}_{j=1}^{\infty}$, and 
$\{ k_j \}_{j=1}^{\infty}$.
\end{proposition}

\begin{proof}
The proof is similar to the one given in \cite[Proposition 3.3]{Rocha2}.
\end{proof}

\begin{lemma} \label{ineq p star}
Let $p(\cdot) : \mathbb{R}^{n} \to (0, \infty)$ be an exponent with $0 < p_{-} \leq p_{+} < \infty$ and let $\{ Q_j \}$ be a family of cubes which satisfies the bounded intersection property. If $0 < p_{*} < \underline{p}$, then
\[
\left\| \left\{ \sum_j \left( \frac{k_j \chi_{Q_j}}{\| \chi_{Q_j} \|_{L^{p(\cdot)}}} \right)^{p_{*}} \right\}^{1/p_{*}}
\right\|_{L^{p(\cdot)}} \approx \mathcal{A} \left( \{ k_j \}_{j=1}^{\infty}, \{ Q_j \}_{j=1}^{\infty}, p(\cdot) \right)
\]
for any sequence of nonnegative numbers $\{ k_j \}_{j=1}^{\infty}$.
\end{lemma}

\begin{proof}
The proof is similar to the one given in \cite[Lemma 5.7]{Rocha2}.
\end{proof}

We are now in a position to give a new estimate of $III$.

\

\underline{\textit{New estimate of $III$}}: Given $1 < q < \infty$ and an exponent $p(\cdot) \in LH_0 \cap L_{\infty}(\mathbb{R}^n)$, let 
$0 < p_{*} < \underline{p}$ be fixed, $p_{0} > \max \{ p_{+}, 2q \}$, and let $a(\cdot)$ be an $(p(\cdot),p_{0},d_{p(\cdot)})$ - atom. So,
\begin{eqnarray*}
\left\| [M(M^{q}(a_j))]^{p_{*}/q} \right\|_{L^{p_{0}/p_{*}}(4\sqrt{n}Q_{j})} &=& 
\left\| [M(M^{q}(a_j))]^{1/q} \right\|_{L^{p_{0}}(4\sqrt{n}Q_{j}))}^{p_{*}} \\
&\lesssim& \left\| M(a_j) \right\|_{L^{p_{0}}(\mathbb{R}^n)}^{p_{*}} \\
&\lesssim& \left\| a_j \right\|_{L^{p_{0}}(\mathbb{R}^n)}^{p_{*}} \\
&\lesssim& \frac{| Q_j |^{\frac{p_{*}}{p_{0}}}}{\left\| \chi _{Q_j }\right\|_{L^{p(\cdot)}}^{p_{*}}} \\
&\lesssim& 
\frac{\left| 4\sqrt{n}Q_{j} \right|^{\frac{p_{*}}{p_{0}}}}{\left\| \chi_{4\sqrt{n}Q_{j}} \right\|_{L^{p(\cdot)/p_{*}}}},
\end{eqnarray*}
where the last inequality follows from Lemma 2.2 in \cite{Nakai}. Now, since $0 < p_{*} < 1$, we apply the $p_{\ast}$-inequality and Proposition \ref{b_k functions} with $b_j = \left( \chi_{4\sqrt{n}Q_{j}} \cdot [M(M^{q}(a_j))]^{p_{*}/q} \right)$, 
$A_j = \left\| \chi_{4\sqrt{n}Q_{j}} \right\|_{L^{p(\cdot)/p_{*}}}^{-1}$ and $s= p_0/p_{*}$, to obtain
\begin{eqnarray*}
\| III \|_{L^{p(\cdot)}} &\lesssim& \left\| \sum_{j} \left(k_j \, \chi_{4\sqrt{n}Q_{j}} \, [M(M^{q}(a_j))] \right)^{p_{*}/q} 
\right\|^{1/p_{*}}_{L^{p(\cdot)/p_{*}}} \\
&\lesssim& \left\| \sum_{j} \left( \frac{k_j}{\left\| \chi_{4\sqrt{n}Q_{j}} \right\|_{L^{p(\cdot)}}} \right)^{p_{*}} 
\chi_{4\sqrt{n}Q_{j}} \right\|^{1/p_{*}}_{L^{p(\cdot)/p_{*}}}.
\end{eqnarray*}
It is easy to check that $\chi_{4\sqrt{n}Q_{j}} \leq [M(\chi_{Q_j})]^{2}$. From this inequality, Lemma 2.2 in \cite{Nakai} 
and Lemma 2.4 in \cite{Nakai}, we have
\begin{eqnarray*}
\| III \|_{L^{p(\cdot)}} &\lesssim& \left\| \left\{ \sum_{j} \left( \frac{k_j^{p_{*}/2}}{\left\| \chi_{Q_{j}} 
\right\|^{p_{*}/2}_{L^{p(\cdot)}}} M(\chi_{Q_j})\right)^{2}  \right\}^{1/2} \right\|^{2/p_{*}}_{L^{2p(\cdot)/p_{*}}} \\
&\lesssim& \left\| \left\{ \sum_{j} \left( \frac{ k_j \chi_{Q_j}}{\| \chi_{Q_j} \|_{L^{p(\cdot)}}} \right)^{p_{*}} 
\right\}^{1/p_{*}} \right\|_{L^{p(\cdot)}}.
\end{eqnarray*}
Finally, Lemma \ref{ineq p star} gives
\[
\| III \|_{L^{p(\cdot)}} \lesssim \mathcal{A}\left( \{ k_j \}_{j=1}^{\infty}, \{ Q_j \}_{j=1}^{\infty}, p(\cdot) \right) 
\lesssim \| f \|_{H^{p(\cdot)}}.
\]

\

Thus, now we have the following improved version of Theorem 1.

\begin{theorem} Let $p(\cdot)$ be an exponent that belongs to $LH_{0}(\mathbb{R}^{n})\cap LH_{\infty}(\mathbb{R}^{n})$, 
$1 < q < \infty$ and $m \in \mathbb{N}$ such that $0 < p_{-} \leq p_{+} < \infty$ and $n (2m + n/q)^{-1} < \underline{p}$. Then 
the operator $\Delta^{m}$ is a bijective mapping from $\mathcal{H}^{p(\cdot)}_{q, 2m}(\mathbb{R}^{n})$ onto 
$H^{p(\cdot)}(\mathbb{R}^{n})$. Moreover, there exist two positive constant $c_1$ and $c_2$ such that
\[
c_1 \|F \|_{\mathcal{H}^{p(\cdot)}_{q, 2m}} \leq \| \Delta^{m}F \|_{H^{p(\cdot)}} \leq c_2 \|F \|_{\mathcal{H}^{p(\cdot)}_{q, 2m}}
\]
hold for all $F \in \mathcal{H}^{p(\cdot)}_{q, 2m}(\mathbb{R}^{n})$.
\end{theorem}

\

Pablo Rocha, Instituto de Matem\'atica (INMABB), Departamento de Matem\'atica, Universidad Nacional del Sur (UNS)-CONICET, Bah\'ia Blanca, Argentina. \\
{\it e-mail:} pablo.rocha@uns.edu.ar

\end{document}